\documentclass[a4paper,12pt]{article}
\usepackage{amssymb,amsthm,amsmath,amsfonts}
\usepackage{graphicx}
\usepackage{cite}
\usepackage{color}

\newtheorem{theorem}{Theorem}[section]

\newtheorem{corollary}[theorem]{Corollary}
\newtheorem{proposition}[theorem]{Proposition}

\theoremstyle{definition}
\newtheorem{definition}[theorem]{Definition}

\theoremstyle{remark}

\begin{document}

\title{Bergman theory for the inhomogeneous Cimmino system}

\author{Jos\'e Oscar Gonz\'alez-Cervantes$^{(1)}$, Dante Arroyo-S\'anchez$^{(1)}$\\ and\\ Juan Bory-Reyes$^{(2)\footnote{corresponding author}}$ }
\small{
\vskip 1truecm
\date{\small $^{(1)}$ Departamento de Matem\'aticas, ESFM-Instituto Polit\'ecnico Nacional. 07338, Ciudad M\'exico, M\'exico\\ Emails: jogc200678@gmail.com, tlacaelela.s68@gmail.com\\$^{(2)}$ {SEPI, ESIME-Zacatenco-Instituto Polit\'ecnico Nacional. 07338, Ciudad M\'exico, M\'exico}\\Email: juanboryreyes@yahoo.com
}}

\maketitle
\begin{abstract}
\small{We first prove a Cauchy's integral theorem and Cauchy type formula for certain inhomogeneous Cimmino system from quaternionic analysis perspective. The second part of the paper directs the attention towards some applications of the mentioned results, dealing in particular with four kinds of weighted Bergman spaces, reproducing kernels, projection and conformal invariant properties.}
\end{abstract}


\section{Introduction}
Gianfranco Cimmino (12 March, 1908 - 30 May, 1989), was the first to consider the homogeneous first order linear partial differential equations (named in his honour {\it Cimmino system}):
\begin{equation}\label{CS}
\begin{cases}
\displaystyle \frac{\partial f_0}{ \partial x_ 0}  + \frac{\partial f_2}{ \partial x_2 } - \frac{\partial f_1}{ \partial x_1 } -\frac{\partial f_3}{ \partial x_3 } =    0  \\
\displaystyle \frac{\partial f_0}{ \partial x_ 1}  + \frac{\partial f_1}{ \partial x_0 } - \frac{\partial f_2}{ \partial x_3 } + \frac{\partial f_3}{ \partial x_2 } =    0   \\
\displaystyle\frac{\partial f_0}{ \partial x_2 }  + \frac{\partial f_3}{ \partial x_1 } - \frac{\partial f_1}{ \partial x_3 } -\frac{\partial f_2}{ \partial x_0 } =    0  \\
\displaystyle\frac{\partial f_0}{ \partial x_3 }  + \frac{\partial f_1}{ \partial x_2 } + \frac{\partial f_2}{ \partial x_1 } +\frac{\partial f_3}{ \partial x_0 } =  0,
\end{cases}
\end{equation}
for $f_m, m= 1, 2, 3, 4$ continuously differentiable $\mathbb R$-valued unknown functions in a domain $\Omega\subset \mathbb R^4$, all of whose solutions are harmonic functions, see \cite{c}.  

The solution theory of the Cimmino system can be seen as a generalization of the complex analysis of two complex variables. In \cite{dl} Dragomir and Lanconelli made a detailed investigation of Cimmino system, succeeded (among many other results) in extend most of the highlights from the holomorphic function theory to the Cimmino system theory, in particular they obtain conditions for the resolvability of the Dirichlet problem for Cimmino system.

Although nowadays the Cimmino system theory has been attacked with greater hope of success, see \cite{ABB, ABGS, Av, BAPS, CC, Xi}, the investigation of inhomogeneous Cimmino system is limited for right-hand sides depending only on the real variables $x_m, m = 1, 2, 3, 4$ , see \cite{W, WXQ}.

The history of Bergman theory goes back to the work of Stefan Bergman (5 May, 1895 – 6 June, 1977) in the early twentieth century, where the first systematic treatment of the subject of reproducing kernel for holomorphic functions of one or several complex variables was establish, see \cite{B} for the classical bibliography.

Since then, efforts were made to develop a similar theory of functions for different linear differential equations in a fixed domain. Some standard works here are \cite{BHZ, DS, HKZ} and the references therein, which contain a broad summary and historical notes of the subject.

At present, systematic investigation of Bergman theory, mainly applied quaternionic methods, is conducted for suitably chosen classes of systems of partial differential equations in a series of papers, see \cite{O, OJ, GM,GLM1,GLM2,GLM3} and references therein.

In this work we prove a Cauchy integral theorem and a Cauchy type formula for the following inhomogeneous Cimmino system: 
\begin{equation}\label{equa1}
\begin{cases}
\displaystyle \frac{\partial f_0}{ \partial x_ 0}  + \frac{\partial f_2}{ \partial x_2 } - \frac{\partial f_1}{ \partial x_1 } -\frac{\partial f_3}{ \partial x_3 } =    q_0    f_0 +  q_2     f_2 -  q_1   f_1 -  q_3   f_3 \\
\displaystyle \frac{\partial f_0}{ \partial x_ 1}  + \frac{\partial f_1}{ \partial x_0 } - \frac{\partial f_2}{ \partial x_3 } +\frac{\partial f_3}{ \partial x_2 } =    q_1   f_0 +  q_0    f_1 -  q_3   f_2 +  q_2     f_3  \\
\displaystyle \frac{\partial f_0}{ \partial x_2 }  + \frac{\partial f_3}{ \partial x_1 } - \frac{\partial f_1}{ \partial x_3 } -\frac{\partial f_2}{ \partial x_0 } =    q_2     f_0 +  q_1   f_3 -  q_3   f_1 -  q_0    f_3 \\
\displaystyle \frac{\partial f_0}{ \partial x_3 }  + \frac{\partial f_1}{ \partial x_2 } + \frac{\partial f_2}{ \partial x_1 } +\frac{\partial f_3}{ \partial x_0 } =     q_3   f_0 +  q_2     f_1 +  q_1   f_2 +  q_0  f_3.
\end{cases}
\end{equation}
for four desired continuously differentiable $\mathbb R$-valued functions $f_m$ depending on four real variables $x_m, m = 1, 2, 3, 4$ and a fix four-tuple $(q_0, q_1, q_2, q_3) \in \mathbb R^4$.

Together with the Cauchy integral theorem and Cauchy type formula versions alluded to earlier, we shall concern ourselves with some applications, in particular with four kinds of weighted Bergman spaces, reproducing functions, projections and conformal invariant properties.

\section{Preliminaries}
The skew-field  of quaternions, customarily denoted by $\mathbb H$ ((in honor of William Rowan Hamilton (August 4, 1805- September 2, 1865)) \cite{H}), is a four dimensional associative algebra over $\mathbb R$ with generators $1$ and the quaternionic imaginary units $e_1, e_2, e_3$ under the multiplication rules:  
\[e_1^2=e_2^2= e_3^2 =-1,  {} e_1e_2=e_3, e_2e_3=e_1,  e_3e_1=e_2.\]
For any $q\in \mathbb H$ it can be written as these linear combinations of $1, e_1, e_2$ and $e_3$, namely $q=q_0+ \vec{q}=q_0 + q_1e_1 + q_2e_2 +q_3e_3$, where $q_k\in\mathbb R$ for $k=0,1,2,3$. Note that a quaternion $q$ may also be viewed as a vector $(q_0, q_1, q_2, q_3) \in \mathbb R^4$. 

The rules for multiplication among $e_1, e_2$ and $e_3$ are enough, with the distributive law, to multiply all quaternions.

Quaternionic conjugation and the module of $q$ are defined respectively by $\bar q :=q_0 - \vec q $ and  $|q|:=\sqrt{ q\bar q }= \sqrt{q_0^2+ q_1^2+  q_2^2 + q_3^2}$.

The dot product of two quaternions $q,r\in\mathbb H$ is given by 
\[\langle  q,r\rangle:=\frac{1}{2}(\bar q r + \bar r q) = \frac{1}{2}(q \bar  r + r  \bar q).\] 

A  four-tuple $\psi=\{\psi_0, \psi_1,\psi_2,\psi_2\}\in\mathbb H^4$ is called structural set if $\langle  \psi_k, \psi_m\rangle  =\delta_{k,m} $, for $k,m=0,1,2,3$ and any quaternion $q$  can be rewritten as  $q_{\psi} := \sum_{k=0}^3 q_k\psi_k$, where $q_k\in\mathbb R$ for all $k$. Note that any structural set can be thought of as an orthonormal (in the usual Euclidean sense) basis in $\mathbb R^4$. The notion of structural set goes back as far as {\cite{No}}. 

Given $q, x\in \mathbb H$ we follow the notation used in {\cite{shapiro1}} to write
$$\langle q, x \rangle_{\psi}=\sum_{k=0}^3 q_k x_k,$$ 
where $q_{\psi} =\sum_{k=0}^3 q_k \psi_k $ and $x_{\psi} =\sum_{k=0}^3 x_k \psi_k$ with $q_k, x_k\in \mathbb R$ for all $k$.

Functions $f$ defined in a bounded domain $\Omega\subset\mathbb H\cong \mathbb R^4$ with value in $\mathbb H$ are considered. They may be written as: $f=\sum_{k=0}^3 f_k \psi_k$, where $f_k, k= 0,1,2,3,$ are $\mathbb R$-valued functions in $\Omega$. Properties as continuity, differentiability, integrability and so on, which as ascribed to $f$ have to be posed by all components $f_k$. We will follow standard notation, for example $C^{1}(\Omega, \mathbb H)$ denotes the set of continuously differentiable $\mathbb H$-valued functions defined in $\Omega$.  

Furthermore, the conformal mappings in $\mathbb R^4$ (also referred to as the M\"obius transformations), are represented by  
\begin{equation}\label{e4}
T(x)=(ax+b)(cx+d)^{-1},
\end{equation}
where  $a\neq 0$, if $c=0$ and $b-ac^{-1}d\neq 0$,  if $c\neq 0$, see \cite{Ah} for more details. 

In 1930's quaternionic functions of one quaternionic variable were investigated mainly by Karl Rudolf Fueter ((30 June, 1880 – 9 August, 1950) \cite{Fu}), who was able to develop the appropriate generalization of the Cauchy-Riemann equations to the quaternionic case. The Cauchy-Fueter equations (as they are now called in honor of their inventor) were not developed until 1935, almost a century after Hamilton's discovery of quaternions. This function theory (called {\it Quaternionic analysis}) has become popular again since 1970's, see e.g. the very well known paper \cite{Su} of  Anthony Sudbery (29 July, 1943 - Unknown). 

Nowadays, quaternionic analysis is regarded as a broadly accepted branch of classical analysis offering a successful generalization of complex analysis. It relies heavily around the concept of $\psi-$hyperholomorphic functions, i.e., null solutions of the so-called $\psi-$Fueter operator related to the structural set $\psi$.
\[{}^\psi\mathcal D= \frac{\partial}{\partial x_0} + \psi_1 \frac{\partial}{\partial x_1} +  \psi_2 \frac{\partial}{\partial x_2} + \psi_3 \frac{\partial}{\partial x_3},\]
which represents a natural generalization of the complex Cauchy-Riemann operator, see \cite{MS, S1, shapiro1} for more details.

In \cite{ABB, ABGS, BAPS}, motivated by \cite[Subsection 6.4-6.5]{W} it is shown how Cimmino system theory can be successfully embedded into that of $\psi$-hyperholomorphic functions for the simplest and therefore fundamental case of structural set 
$$\psi=\{1,e_1,-e_2,e_3\}.$$ 

The theory of $\psi$-hyperholomorphic functions has recently rekindled by the study of a quaternionic valued function theory associated to a domain $\Omega\subset \mathbb H$ and induced by the left- and the right-$ q $-$\psi$-Fueter type operators:   
$${}^{{\psi}}_{{q}}\mathcal D [f] := {}^{{\psi}}\mathcal D[f] +  q f \quad  \textrm{and}   \quad  \mathcal D^{{\psi}}_{{q}} [f] := \mathcal D^{{     \psi}}[f] + f  q, \quad f\in C^1(\Omega,\mathbb H),$$ 
which the first author of the present paper has studied in some detail \cite{O}. 

We will end this section by recalling relevant material from \cite{O} and provide some additional notations and terminology that will be used in the sequel, thus making our exposition self-contained. To this end, let us introduce the following notations: 
\begin{itemize}
\item $d\lambda^{\psi}_{q}:=d\lambda_{ q  }^{\psi}(x)=e^{2  \langle q, x \rangle_{\psi}} d\mu_x$,

\item $d{\mu}_x$ denotes the oriented volume element on $\Omega$

\item $\nu_{q}^\psi(x):= e^{2 \langle q, x \rangle_{\psi}} \sigma_x^{\psi}$,

\item $\sigma_x^{\psi}:= -sgn\psi \left( \sum_{k=0}^3 (-1)^k \psi_k d\hat{x}_k\right)$ stands for the quaternionic differential form of the 3 dimensional volume in $\mathbb R^4$ according  to $\psi$, with $d\hat{x}_i= dx_0 \wedge dx_1\wedge dx_2  \wedge  dx_3 $ omitting factor $dx_i$,

\item $sgn\psi$ is equal to $1$, or $-1$, provided that $\psi$ and $(1,e_1,e_2,e_3)$ have the same orientation or not.

\end{itemize}

In addition, we write $K^{\psi}_{ q  } (y-x) := e^{\langle q, y-x \rangle_{\psi}} K_{\psi} (y-x)$, where 
\begin{equation*}   
K_{\psi}(\tau- x ):=\frac{    1}{2\pi^2} \frac{ \overline{\tau_{\psi} - x_{\psi}}}{|\tau_{\psi} - x_{\psi}|^4},
\end{equation*}
for the $\ell$-$ q  $-$\psi$-hyperholomor\-phic Cauchy Kernel. 

In particular, when $\psi=\{1,e_1,-e_2,e_3\}$ we denote it briefly by 
\begin{align}\label{equa56}
 K_{q} (y-x) =  &  e^{  \langle  q   , y-x \rangle } K  (y-x) \nonumber \\ 
  = &  K^q_0(y-x) + K^q_1(y-x) e_1+ K^q_2(y-x) e_2 + K^q_3(y-x) e_3  \nonumber  \\ 
	=  &
  \left( K^q_0(y-x)  ,  K^q_1(y-x) ,  K^q_2(y-x) ,  K^q_3(y-x) \right).
\end{align}

\begin{proposition}\label{proplambda}
Let $\Omega\subset \mathbb H$ be a domain with a 3-dimensional compact smooth boundary, denoted by $\partial \Omega$ and fix $q\in\mathbb H$. Then 
\begin{align}
d(f  \nu_{ q  }^\psi(x)  g) = & ( \mathcal D^{\psi}_{ q  }[f] g  +   f {}^{\psi}_{ q  }\mathcal D[g] ) e^{2  \langle  q  ,x \rangle_{\psi}}, \nonumber\\
\label{STOkes_alpha_psi}\int_{\partial \Omega} f(x) \nu_{ q  }^\psi(x)  g(x) = &\int_{\Omega} ( \mathcal D^{\psi}_{ q  }[f](y) g(y)  +   f(y) {}^{\psi}_{ q  }\mathcal D[g](y) ) d\lambda_{ q  }^{\psi}(y) 
\end{align}
and
\begin{align}  &  \int_{\partial \Omega}  ( K^{\psi}_{ q  }( y -x)\sigma_{ y }^{\psi} f( y )  +  g( y ) \sigma_{ y }^{\psi} K^{\psi}_{ q  }( y -x) ) \nonumber  \\ 
&  - 
\int_{\Omega} (K^{\psi}_{ q  } (y-x)
  {}^{\psi}_{ q  }\mathcal D [f] (y)  +     \mathcal D^{\psi}_{ q  } [g] (y)   K^{\psi}_{ q  } (y-x)
     )d   {\mu}_y   \nonumber \\
		=  &    \label{ecuaNEWCauchy}  \left\{ \begin{array}{ll}  f(x) + g(x) , &  x\in \Omega,  \\ 0 , &  x\in \mathbb H\setminus\overline{\Omega},                    
\end{array} \right. 
\end{align} 
for all $f,g \in C^1(\overline{\Omega}, \mathbb H)$.
\end{proposition}

\begin{proposition}\label{conformal}
Let $q, r \in\mathbb H$ and $\psi$ be a structural set. Suppose $\Omega,\Xi \subset \mathbb H$ be  two conformally equivalent domains such that $T(\Xi)=\Omega$ and set $y=T(x)$. Then  
\begin{align}\label{eq7}
{}^{\psi}_{ q }\mathcal D_x \circ [e^{  \langle   r   -  q  , x \rangle_{\psi}} A_T f\circ T ] =   
  e^{  \langle   r   -  q  , x \rangle_{\psi}}   (B_T\circ T)  (   {}^{\psi}_{\delta_{T,r}} \mathcal D_y [   f] ) \circ T ,
\end{align}
where 
\begin{align*}  A_{T}(x):= & \left\{ \begin{array}{l}
\bar d, \quad  \textrm{if {} }  c=0; \\
\\
  {
  \overline{V} \dfrac{\overline{c x V^{-1} + dV^{-1}
}}{|c x V^{-1} + dV^{-1}|^{4}} }, \quad
   \textrm{if {} } c\neq  0,
 \end{array}\right.  \quad x\in \Omega
 \end{align*}

  \begin{align*}	B_{T}(y):= &  \left\{
\begin{array}{l}
\bar a, \quad  \textrm{if {} }  c=0; \\
\\
  {- \bar c \overline{(y -ac^{-1})} |y -ac^{-1}|^{4}  },
    \quad \textrm{if {} } c\neq 0,
 \end{array}\right. \quad y\in \Xi,\\ 
\delta_{T,r}(y):= & \left\{ \begin{array}{ll}
 (\bar a )^{-1} r   \bar d, &   \textrm{if {} }  c=0; \\
\\  
\\   
- \dfrac{ (y-ac^{-1})c  r   \bar{ V} (y-ac^{-1})  }{  |c|^2 |y-ac^{-1}|^4 }, &  
   \textrm{if {} } c\neq  0,
 \end{array}\right. \quad  \forall y\in \Xi, 
\end{align*} 
where  
  $V=b-ac^{-1}d$  and  $y=T(x)$ for $x\in \Omega$. 
	\end{proposition}
Under the conditions stated above let us introduce two more piece of notation: 
	\begin{align} \label{coef1} C_T(x):= & \left\{
\begin{array}{ll}
   \dfrac{|a|^2}{|d|}d^{-1}, &   \textrm{ if  $c=0$, }
 \\
 \\
     \dfrac{|c|^2}{|V|}V^{-1}
   \dfrac{\overline{cxV^{-1}+ dV^{-1}}}{|cxV^{-1}+ dV^{-1}|^4},  & \textrm{ if $c\neq 0$, }
  \end{array} \right. \\
	\nonumber \\
\label{peso1}  \rho_{_T}(x):= & \left\{
\begin{array}{ll}
1,  &  \textrm{if  $c=0$,}
\\
\\
  \dfrac{1}{|c x V^{-1} + dV^{-1}|^2}, & 
 \textrm{ if   $c\neq 0$.}
\end{array} \right.
\end{align} 
 
\begin{definition}\label{deflambda}	Let $\Omega\subset \mathbb H$ be a domain. Let $\psi$ be a structural set and $q\in\mathbb H$. 
\begin{enumerate}
\item  The quaternionic right-linear space ${}^{\psi}_{q}\mathcal A(\Omega)= Ker {}^{\psi}_{q}\mathcal D \ \cap \ C^1(\Omega, \mathbb H) \ \cap \ \mathcal L_2(\Omega, \mathbb H)$ is called $\ell$-$ q  $-$\psi$-hyperholomorphic Bergman space associated to $\Omega$ and denote  
\begin{align*}
\|f\|_{{}^{\psi}_{q}\mathcal A(\Omega)}  = \left( \int_{\Omega}|f|^2d\mu  \right)^{\frac{1}{2}},
\end{align*}
for all $f\in {}^{\psi}_{q}\mathcal A(\Omega)$. By  $\mathcal A^{\bar \psi}_{\bar  q}(\Omega)= \mathfrak M^{\bar \psi}_{\bar q}(\Omega)  \cap \mathcal L_2(\Omega, \mathbb H)$ we mean the right-$\bar q  $-$\bar \psi$-hyperholomorphic Bergman space associated to $\Omega$. 
\item  The set  ${}_{{\lambda}_{ q  }^\psi}\mathcal A  (\Omega)$ consists of 		$f\in  {}^\psi_{ q  }\mathfrak M(\Omega) $ such that 
\begin{align*}
\|f\|_{{}_{{\lambda}_{ q  }^\psi}\mathcal A  (\Omega)}^2: =  \int_{\Omega} |f(x)|^2 d\lambda^{\psi}_{ q  }(x) <\infty.
\end{align*}
Define  			    
\begin{equation*}
\left\langle f,g\right\rangle_{{}_{{\lambda}_{q}^\psi}\mathcal A(\Omega)}:= \int_{\Omega} \bar f g d\lambda^{\psi}_{ q  }, \quad \quad
{}_{q}^\psi\mathcal S [f](x):= e^{\left\langle q, x \right\rangle_{\psi}}f(x), \quad  x\in \Omega, 
\end{equation*} 
for all $f,g\in {}_{{\lambda}_{ q  }^\psi}\mathcal A (\Omega)$.  

Clearly,  
${}_{\lambda_{0}^\psi}\mathcal A(\Omega) = {}_{0}^\psi\mathcal A(\Omega)= {}^\psi\mathcal A(\Omega)$, as quaternionic right-Hilbert spaces. 
\end{enumerate} 
\end{definition}

\begin{proposition}\label{HilertSpace} 
Let $\Omega\subset \mathbb H$ be a domain and set $q, r \in \mathbb H$. 
\begin{enumerate}
\item  The quaternionic right-linear space
$({}^{\psi}_{q}\mathcal A(\Omega), \  \langle  \cdot , \cdot  \rangle_{{}^{\psi}_{q}\mathcal A(\Omega)})$ is a quaternionic right-Hilbert space and 
\begin{align}\label{BergKern}f(x)=\int_{\Omega} {}^{\psi}_{q}\mathcal B_{\Omega}(x,\zeta )f(\zeta)d\mu_{\zeta}, \quad 
 f\in {}^{\psi}_{q}\mathcal A(\Omega),
\end{align} 
where ${}^{\psi}_{q}\mathcal B_\Omega(x,\cdot) $ is the Bergman kernel of ${}^{\psi}_{q}\mathcal A(\Omega)$. Its Bergman projection is given by          
\begin{align}\label{BergProj}  
{}^{\psi}_{q} \mathfrak B_{\Omega }[f](x):=\int_{\Omega}
{}^{\psi}_{q}\mathcal B_{\Omega }(x,\zeta) f(\zeta) d\mu_\zeta, \quad f\in\mathcal L_{2}(\Omega,\mathbb H).
\end{align}  
\item ${}_{ q-r }^{\psi} S:  {}_{   \lambda   _{ q }^\psi} \mathcal A(\Omega)\to  {}_{   \lambda   _{ r  }^\psi} \mathcal A(\Omega) $ is an isometric isomorphisms of quaternionic right-Hilbert spaces given by 
$${}_{ q-r }^{\psi} S [f](x) = e^{\langle q-r, x \rangle } f(x) ,\quad x\in \Omega$$
and 
\begin{align*} {}_{\lambda_{r}^\psi}\mathcal  B_{\Omega}(x,\xi)= &  e^{ \langle  q - r  , x +\xi  \rangle_{\psi}} {}_{ \lambda_{ q }^\psi}\mathcal B_{\Omega}(x,\xi) ,
\end{align*}
where  
${}_{\lambda_{r}^\psi}\mathcal  B_{\Omega}$ is the kernel of ${}_{\lambda_{r}^\psi} \mathcal A(\Omega)$.
\end{enumerate}
\end{proposition}

\begin{proposition}\label{ref123456} 
Given $q, r\in\mathbb H$  and let $\psi \subset\mathbb H$ be a structural set. Assume $\Omega$ and $\Xi$ be two conformally equivalent domains with  $T(\Xi)= \Omega $. 

Let us define  $\gamma_{_{T,r,q}}(y)= e^{2\langle  r  - q   , T^{-1}(y) \rangle_{\psi}}$ for all $y\in \Omega$ and the function space 
$${{}^{\psi}_{\delta_{T,r}} \mathcal A_{\gamma_{_{T,r,q}}} (\Omega) } :=  Ker {}^{\psi}_{\delta_{T,r}}\mathcal D \  \cap \ C^1(\Omega, \mathbb H) \ \cap \ \mathcal L_{2,\gamma_{_{T,r,q}}}(\Omega, \mathbb H),$$  
where 
$$ {}^{\psi}_{\delta_{T,r}}\mathcal D [f]:= {}^{\psi}\mathcal D [f]+  {\delta_{T,r}}f, \quad  f\in C^1(\Omega, \mathbb H).$$
The function space ${{}^{\psi}_{\delta_{T,r}} \mathcal A_{\gamma_{_{T,r,q}}} (\Omega)}$ is equipped with the norm and the inner product inherited from the weighted space $\mathcal L_{2,\gamma_{_{T,r,q}}} (\Omega,\mathbb H)$: 
\begin{equation*} 
\|f\|_{{}^{\psi}_{\delta_{T,r}} \mathcal A_{\gamma_{_{T,r,q}}} (\Omega)}:= \left[\int_{\Omega} \|f\|^2 \gamma_{_{T,r,q}}d\mu\right]^{\frac{1}{2}}, \ \ \langle f,g \rangle_{{}^{\psi}_{\delta_{T,r}}\mathcal A_{\gamma_{_{T,r,q}}}(\Omega)}:= \int_{\Omega} \bar f g \gamma_{_{T,r,q}} d\mu,
\end{equation*}
for all $f,g \in {}^{\psi}_{\delta_{T,r}} \mathcal A_{\gamma_{_{T,r,q}}}(\Omega)$. 

On the other hand, set $ {}^{\psi}_{q} \mathcal A_{\rho_{_T}} (\Xi)  :=  Ker {}^{\psi}_{q}\mathcal D \ \cap \ C^1(\Xi, \mathbb H)  \  \cap  \ \mathcal L_{2,\rho_{_T} }(\Xi, \mathbb H)$. Then 
\begin{equation*}   
\|f\|_{{}^{\psi}_{q} \mathcal A_{\rho_{_T}}(\Xi)} =  \left[ \int_{\Xi} \|f\|^2 \rho_{_T} d\mu\right]^{\frac{1}{2}}, \ \ \langle f,g \rangle_{{}^{\psi}_{q} \mathcal A_{\rho_{_T}} (\Xi)} =  \int_{\Xi} \bar f g  \rho_{_T} d\mu,
\end{equation*}
for all $f,g \in {}^{\psi}_{q} \mathcal A_{\rho_{_T}} (\Xi)$. 

The quaternionic right-linear operator $\mathcal J:{{}^{\psi}_{\delta_{T, r}} \mathcal A_{\gamma_{_{T,r,q}}}(\Omega) } \to {}^{\psi}_{q} \mathcal A_{\rho_{_T}} (\Xi)$ given by
\begin{align}\label{ConfOper} 
\mathcal J[f](x)=  e^{\langle  r  - q  , x \rangle_{\psi}}C_T (x) f\circ T(x), \quad  x\in \Xi,
\end{align}
for all $f\in {{}^{\psi}_{\delta_{T,r}} \mathcal A_{\gamma_{_{T,r,q}}}(\Omega)}$ is  an isometric isomorphism of quaternionic right-Hilbert spaces. What is more, 
  \begin{align*}  {}^\psi_{ q  }\mathcal  B_{\Xi, \rho_{_T}}(x,\xi)= & e^{\langle  r  - q  , x+ \xi \rangle_{\psi}}    C_T(x) \ {}^\psi_{\delta_{T,r}}\mathcal B_{\Omega,{\gamma_{_{T,r,q}}}}(T(x), T(\xi))
   \overline{C_T(\xi)},
\end{align*}
where ${}^\psi_{q}\mathcal B_{\Xi, \rho_{_T}}$ and ${}^\psi_{\delta_{T,r}}\mathcal B_{\Omega,{\gamma_{_{T,r,q}}}}$ stand for the Bergman kernels of ${}^{\psi}_{q} \mathcal A_{\rho_{_T}}(\Xi)$ and ${ {}^{\psi}_{\delta_{T,r}} \mathcal A_{\gamma_{_{T,r,q}}} (\Omega)}$, respectively.
\end{proposition}

\begin{corollary}\label{relationships} Let $\Omega\subset\mathbb H$ be a domain. 
\begin{enumerate}
\item If $\Omega \subset \{ x \in \mathbb H \  \mid \  \left\langle q  , x \right\rangle_{\psi} <0  \}$
then $  {}_{ q  }^\psi\mathcal A (\Omega)\subset {}_{\lambda_{ q  }^\psi}\mathcal A (\Omega) $ as function sets. 
\item  If $\Omega \subset \{ x \in \mathbb H \  \mid \  \left\langle q  , x \right\rangle_{\psi} > 0 \}$
then   $  {}_{   \lambda   _{ q  }^\psi}\mathcal A (\Omega)  \subset   {}_{ q  }^\psi\mathcal A (\Omega) $ as function sets. 
\item If $\Omega$ is a bounded domain then  ${}_{\lambda_{q}^\psi}\mathcal A (\Omega)= {}_{q}^\psi \mathcal A (\Omega)$
as function sets.  
\end{enumerate}
\end{corollary}

\section{Main results}
It is illuminating worth noting that while the equation ${}^{\psi}\mathcal D[f]=0$, with $\psi=\{1,e_1,-e_2,e_3\}$, is equivalent to the Cimmino system \eqref{CS}, the inhomogeneous Cimmino system \eqref{equa1} is obtained from ${}^{\psi}_{q}\mathcal D[f]=0 $ with $q= q_0 + q_1 e_1 - q_2 e_2 + q_3 e_3\in \mathbb H$ for the same structural set, so they are of great importance for the purposes of this article.

Throughout what follows, the main attention is given to the singular and important case of structural set $\psi=\{1,e_1,-e_2,e_3\}$.

The main results of the presented paper are formulated in the following theorems.

\begin{theorem}\label{CauchyTypeTeorem} (Cauchy's integral theorem)
Let $\Omega\subset \mathbb R^4$ be a domain with a 3-dimensional compact smooth boundary, denoted by $\partial \Omega$. If $f =(f_0,f_1,f_2,f_3)$ sa\-tis\-fies \eqref{equa1} then

\begin{align*} 
     \int_{\partial \Omega}     e^{2\langle q, x \rangle } \left[    d\hat{x}_0   f_0 (x) +  d\hat{x}_1  f_1(x) + d\hat{x}_2   f_2(x)  +  d\hat{x}_3 f_3(x)  \right]   = & 0  , \\ 
\int_{\partial \Omega}    e^{2\langle q, x \rangle } \left[ - d\hat{x}_1  f_0(x)  +  d\hat{x}_0   f_1(x) +  d\hat{x}_3 f_2(x)   - d\hat{x}_2   f_3(x)  \right]  = & 0, \\
\int_{\partial \Omega}      e^{2\langle q, x \rangle }\left[   - d\hat{x}_2   f_0(x) +  d\hat{x}_1  f_3(x) + d\hat{x}_3 f_1 (x) +  d\hat{x}_0   f_2 (x)  \right]   = & 0, \\ 
 \int_{\partial \Omega}     e^{2\langle q, x \rangle }\left[  -  d\hat{x}_3 f_0(x) + d\hat{x}_2   f_1(x) -  d\hat{x}_1  f_2(x) +  d\hat{x}_0   f_3(x)   \right]   = & 0 .\end{align*}
\end{theorem}

\begin{proof} For $\psi=\{1,e_1,-e_2,e_3\}$ we see that sig$\psi=-1$ and $f =(f_0,f_1,f_2,f_3)$ satisfy \eqref{equa1} if and only if $f = f_0 + f_1e_1 + f_2e_2 + f_3e_3$ belongs to Ker ${}^{\psi}_{q}\mathcal D\cap C^1(\Omega, \mathbb H)$. 

By \eqref{STOkes_alpha_psi} and taking $g=1$ we obtain  
\begin{align}\label{CauchyTypeTheorem}
\int_{\partial \Omega} \nu^{\psi}_q(x)  f(x) = 0. 
\end{align}
The real components of the previous identity establishes the formulas
\end{proof}

According to \eqref{equa56} with the notation
\begin{equation*}
\begin{split}
K^q_{\sigma, 0}(y-x)~=~K^q_0(y-x) d\hat{y}_0 + K^q_1(y-x)d\hat{y}_1 + K^q_2(y-x)d\hat{y}_2 + K^q_3(y-x)d\hat{y}_3, \\ 
K^q_{\sigma, 1}(y-x)=-K^q_0(y-x) d\hat{y}_1 + K^q_1(y-x)d\hat{y}_0 - K^q_2(y-x)d\hat{y}_3+ K^q_3(y-x)d\hat{y}_2, \\
K^q_{\sigma, 2}(y-x)=-K^q_0(y-x) d\hat{y}_2 + K^q_1(y-x)d\hat{y}_3 + K^q_2(y-x)d\hat{y}_0- K^q_3(y-x)d\hat{y}_1,  \\
K^q_{\sigma, 3}(y-x)=-K^q_0(y-x) d\hat{y}_3 - K^q_1(y-x)d\hat{y}_2 + K^q_2(y-x)d\hat{y}_1+ K^q_3(y-x)d\hat{y}_0. 
\end{split}
\end{equation*}
we have
\begin{theorem}\label{CauchyTypeFormula} (Cauchy type formula) 
Let $\Omega\subset \mathbb R^4$ be a domain with a 3-dimensional compact smooth boundary denoted by $\partial \Omega$. 
If $f =(f_0,f_1,f_2,f_3)$ satisfies \eqref{equa1} then
\begin{align*}    
& \int_{\partial \Omega}  K^q_{\sigma,0} (y-x)f_0(y)-K^q_{\sigma,1}(y-x)f_1(y)-K^q_{\sigma,2}(y-x)f_2(y)-K^q_{\sigma,3}(y-x)f_3(y) \\
=  & \left\{ \begin{array}{ll}  f_0(x)   , &  x\in \Omega,  \\ 0 , &  x\in \mathbb H\setminus\overline{\Omega} ,                    \end{array} \right. \\
  &  \int_{\partial \Omega }  
  K^q_{\sigma, 0} (y-x)   f_1(y) +  K^q_{\sigma, 1} (y-x)  f_0(y)  + K^q_{\sigma, 2} (y-x)  f_3(y) - 
 K^q_{\sigma, 3} (y-x)   f_2 (y ) 
\\
   		=  &      \left\{ \begin{array}{ll}  f_1(x)   , &  x\in \Omega,  \\ 0 , &  x\in \mathbb H\setminus\overline{\Omega} ,                    \end{array} \right. 
\\
  &  \int_{\partial \Omega }  
  K^q_{\sigma, 0} (y-x)   f_2(y) -  K^q_{\sigma, 1} (y-x)  f_3(y)  + K^q_{\sigma, 2} (y-x)  f_0(y) - 
 K^q_{\sigma, 3} (y-x)   f_1 (y ) \\
   		=  &      \left\{ \begin{array}{ll}  f_2(x)   , &  x\in \Omega,  \\ 0 , &  x\in \mathbb H\setminus\overline{\Omega} ,                    \end{array} \right. 
\\
  &  \int_{\partial \Omega }  
  K^q_{\sigma, 0} (y-x)   f_3(y) +  K^q_{\sigma, 1} (y-x)  f_2(y)  - K^q_{\sigma, 2} (y-x)  f_1(y) + 
 K^q_{\sigma, 3} (y-x)   f_0(y ) \\
   		=  &      \left\{ \begin{array}{ll}  f_3(x)   , &  x\in \Omega,  \\ 0 , &  x\in \mathbb H\setminus\overline{\Omega}.     
			\end{array} \right. 
		\end{align*} 
		
\end{theorem}
\begin{proof} Set $\psi=\{1,e_1,-e_2,e_3\}$. Then 
\begin{align*}
K_{q}(y-x)\sigma_{y}^{\psi} =  & 
\left[ K^q_0(y-x) + K^q_1(y-x) e_1+ K^q_2(y-x) e_2 + K^q_3(y-x) e_3\right] \\
 &  \ \  \  \  \  (    d\hat{y}_0 -   d\hat{y}_1 e_1 -   d\hat{y}_2 e_2  -   d\hat{y}_3  e_3 ) \\
=  & K^q_{\sigma, 0} (y-x)  
 + K^q_{\sigma, 1} (y-x) e_1 +K^q_{\sigma, 2} (y-x) e_2 + K^q_{\sigma, 3} (y-x) e_3,
\end{align*}
where $\sigma_{y}^{\psi}:= d\hat{y}_0 - e_1 d\hat{y}_1 -  e_2 d\hat{y}_2 -  e_3 d\hat{y}_3.$

Due to $f =(f_0,f_1,f_2,f_3)$ satisfy \eqref{equa1} if and only if $f = f_0 + f_1e_1 + f_2e_2 + f_3e_3 \in $ Ker ${}^{\psi}_{q}\mathcal D\cap C^1(\Omega, \mathbb H)$ we can use \eqref {ecuaNEWCauchy} and $g\equiv 0$ to obtain  
\begin{align*}    \int_{\partial \Omega }   K^{\psi}_{ q  }( y -x)\sigma_{ y }^{\psi} f( y ) =  &  \left\{ \begin{array}{ll}  f(x), &  x\in \Omega,  \\ 0 , &  x\in \mathbb H\setminus\overline{\Omega} ,                    
\end{array} \right. 
\end{align*} 
and finally, we find its real components and the proposition follows.
\end{proof}

\begin{proposition}\label{conformalCimmino}
Given $q=(q_0,q_1,q_2,q_3),  r= (r_0,r_1,r_2,r_3)\in \mathbb R^4$ and  let $\Omega,\Xi \subset \mathbb H$ be two conformally equivalent domains such that  $T(\Xi)=\Omega$.Then $(f_0,f_1 f_2,f_3)$ defined on $\Omega $ satisfies the inhomogeneous Cimmino system  
	\begin{align}\label{CimminononconstantPerturba}
     \frac{\partial f_0 }{\partial x_0}  + \frac{\partial f_2}{\partial x_2}  - \frac{\partial f_1}{\partial x_1}  - \frac{\partial f_3}{\partial x_3}  = & \ \delta_{T,r}^0 f_0 +\delta_{T,r}^2 f_2 -\delta_{T,r}^1 f_1 -\delta_{T,r}^3 f_3 \nonumber \\
     \frac{\partial f_0}{\partial x_1}  +\frac{\partial f_1}{\partial x_0}  -\frac{\partial f_2}{\partial x_3}  +\frac{\partial f_3}{\partial x_2}  = &  \ \delta_{T,r}^1 f_0+\delta_{T,r}^0 f_1-\delta_{T,r}^3f_2 + \delta_{T,r}^2 f_3   \nonumber \\
      \frac{\partial f_0}{\partial x_2}  +\frac{\partial f_3}{\partial x_1}  -\frac{\partial f_1}{\partial x_3}  -\frac{\partial f_2}{\partial x_0}  = &  \  \delta_{T,r}^2 f_0 +\delta_{T,r}^1f_3 -\delta_{T,r}^3 f_1 -\delta_{T,r}^0 f_2 \nonumber  \\
      \frac{\partial f_0}{\partial x_3}  + \frac{\partial f_1}{\partial x_2}  +\frac{\partial f_2}{\partial x_1}  +\frac{\partial f_3 }{\partial x_0}  =& \ \delta_{T,r}^3f_0+\delta_{T,r}^2f_1+\delta_{T,r}^1f_2 +\delta_{T,r}^0f_3, 
\end{align}
where  $\delta_{T,r} = \delta_{T,r}^0 + \delta_{T,r}^1 e_1 - \delta_{T,r}^2 e_2 + \delta_{T,r}^3 e_3$ and $\delta_{T,r}^n$ is real-valued for $n=0,1,2,3$,  if and only if the four-tuple of functions 
\begin{align*}
  g_0(x) = &e^{  \langle   r   -  q  , x \rangle } \left( A_T^0 (x)f_0 (x)-A_T^1(x) f_1(x) -A_T^2 (x)f_2(x) -A_T^3(x) f_3(x)\right),\\
g_1 (x)= & e^{  \langle   r   -  q  , x \rangle } \left(  A_T^1 (x) f_0(x)+A_T^0(x) f_1(x)-A_T^3 (x) f_2(x) + A_T^2(x) f_3(x) \right),\\
   g_2 (x)   =  & e^{  \langle   r   -  q  , x \rangle } \left( A_T^2 (x) f_0(x) - A_T^1 (x)f_3(x) + A_T^3 (x) f_1(x) +A_T^0(x)f_2(x)\right),\\
   g_3  (x) = &  e^{  \langle   r   -  q  , x \rangle}\left(  A_T^ 3  (x) f_0(x)-A_T^2 (x) f_1(x)+A_T^1 (x) f_2(x) +A_T^0 (x) f_3(x)\right) ,  \quad x\in \Omega, 
\end{align*}
where $A_T = A_T^0 + A_T^1 e_1 + A_T^2 e_2 + A_T^3 e_3$ and $A_T^n$ is a real-valued for $n=0,1,2,3$, satisfies \eqref{equa1}. 
\end{proposition} 

\begin{proof}
Consider  $\psi= \{1,e_1,-e_2,e_3\}$. Due to $f = f_0 + f_1e_1 + f_2e_2 + f_3e_3$ belongs to Ker ${}^{\psi}_{\delta_{T,r}}\mathcal D\cap C^1(\Omega, \mathbb H)$ we can use \eqref{eq7} and find its real components.  
\end{proof}

\begin{definition} Let $\Omega\subset\mathbb R^4$ be a domain and fix $(q_0,q_1,q_2,q_3)\in \mathbb R^4$.
\begin{enumerate} 
\item The Bergman space associated to $\Omega$ induced by \eqref{equa1}, denoted by $CB_q(\Omega)$, consists of the four-tuples  $(f_0 , f_1, f_2, f_3 )$ of elements of $C^1(\Omega,\mathbb R)$ that satisfy \eqref{equa1} and $\displaystyle \int_{\Omega} |f_k|^2  d\mu <\infty, \quad  k=0,1,2,3$. 

Define the real-bilinear form 
$$\langle \cdot, \cdot \rangle_{CB_q(\Omega)}: CB_q(\Omega) \times CB_q(\Omega) \to \mathbb R^4$$
as follows: 
\begin{align*} 
& \langle  f, 
 g \rangle_{CB_q(\Omega)} \\ 
= & \left(  \int_{\Omega} (   f_0g_0 +  f_1g_1 +  f_2 g_2  +  f_3g_3  ) d\mu , \ 
\int_{\Omega} (  f_0g_1  -   f_1g_0 + f_2  g_3-   f_3g_2  ) d\mu  , \right. \\
  &  \left. \ 
\int_{\Omega} (   f_0 g_2  +  f_3  g_1- f_1 g_3 -   f_2 g_0  ) d\mu   , \ 
 \int_{\Omega} (   f_0 g_3 -  f_1 g_2 +   f_2g_1 -    f_3g_0   ) d\mu \right) ,\end{align*}
and  set
$$ 
\| f\|_{CB_q(\Omega)} =\left[ \int_{\Omega} (f_0^2 +f_1^2+f_2^2 + f_3^2 )  d\mu\right]^{\frac{1}{2}} ,$$ for all 
  $f =    (f_0,f_1, f_2 , f_3) , \ g 
 = (  g_0 , g_1 ,  g_2  ,  g_3 ) \in    CB_q(\Omega)$. 

Note that $ \|f\|_{CB_q(\Omega)}^2 = \Pi_0(\langle  f, f \rangle_{  CB(\Omega)})$ where $\Pi_0:\mathbb R^4 \to \mathbb R$ denotes the projection $\Pi_0(x_0,x_1,x_2,x_3)= x_0$ for all  $(x_0,x_1,x_2,x_3) \in \mathbb R^4$. 
 
\item The weighted Bergman space with wight function $e^{2  \langle q, x\rangle} $ for all $x\in\Omega$ and induced by the inhomogeneous Cimmino system \eqref{equa1} is denoted by $CB_{\lambda_q}(\Omega)$  and it is formed by $(f_0,f_1,f_2,f_3)$, solutions \eqref{equa1} such that 
$$\int_{\Omega} |f_k|^2 d\lambda_q    =  \int_{\Omega} |f_k(x)|^2  e^{2  \langle q, x\rangle} d\mu_x   <\infty, \quad  k=0,1,2,3.$$
Given  
$f = (f_0,f_1, f_2 , f_3) , \ g = (g_0 , g_1 ,  g_2  ,  g_3 ) \in    CB_{\lambda_q}(\Omega)$   define the real-bilinear form  
$\langle  \cdot , \cdot  \rangle_{{CB}_{\lambda_q}(\Omega)} : {CB}_{\lambda_q}(\Omega)\times {CB}_{\lambda_q}(\Omega) \to \mathbb R^4$ given by 
\begin{equation*}
\begin{split} 
\langle f, g \rangle_{CB_{\lambda_q}(\Omega)} =  \\
\left(\int_{\Omega} (f_0g_0 + f_1g_1 + f_2 g_2  +  f_3g_3)d\lambda_q , \int_{\Omega} (f_0g_1 - f_1g_0 + f_2  g_3- f_3g_2)d\lambda_q, \right. \\
\left. \int_{\Omega} (f_0g_2 + f_3g_1 - f_1 g_3 - f_2 g_0 )d\lambda_q, \int_{\Omega} (f_0 g_3 - f_1 g_2 + f_2g_1 - f_3g_0)d\lambda_q \right)
\end{split}
\end{equation*}
and consider  
\begin{equation*} 
\begin{split}
\|f\|_{CB_{\lambda_q}(\Omega)} = \|(f_0,f_1, f_2 , f_3)\|_{CB_{\lambda_q}(\Omega)}:=\left[\int_{\Omega}(f_0^2 +f_1^2+f_2^2 + f_3^2)d\lambda_q \right]^{\frac{1}{2}}.
\end{split}
\end{equation*}
Note that 
\begin{equation*} 
\begin{split}
\|f\|^2_{CB_{\lambda_q}(\Omega)} = \Pi_0(\langle f, f\rangle_{CB_{\lambda_q}(\Omega)}),
\end{split}
\end{equation*}
for all $f\in CB_{\lambda_q}(\Omega)$.
\item Let $q, r\in\mathbb R^4$ and let $\Omega$ and $\Xi$ be two conformally equivalent domains such that $T(\Xi)=\Omega$. According to Proposition \ref{conformalCimmino} set $\gamma_{_{T,r,q}}(y)= e^{2\langle r - q, T^{-1}(y) \rangle}$ for all $y\in \Omega$. Let $CB^{\delta_{T,r}}_{\gamma_{_{T,r,q}}}(\Omega)$ denote the set of four-tuples of functions in $C^1(\Omega, \mathbb R)$ and  let $(f_0,f_1,f_2,f_3)$ defined on $\Omega$ satisfying \eqref{CimminononconstantPerturba} and such that
$$\displaystyle \int_{\Omega} |f_k|^2 \gamma_{_{T,r,q}} d\mu <+\infty,$$ 
for $k=0,1,2,3$.   
	
For $f=(f_0,f_1,f_2,f_3), g=(g_0,g_1,g_2,g_3) \in CB^{\delta_T}_{\gamma_{_{T,r,q}}} (\Omega)$  define the real-bilinear form 
$$\langle \cdot,\cdot \rangle_{CB^{\delta_{T,r}}_{\gamma_{_{T,r,q}}} (\Omega)}: CB^{\delta_{T,r}}_{\gamma_{_{T,r,q}}} (\Omega)  \times CB^{\delta_{T,r}}_{\gamma_{_{T,r,q}}} (\Omega)  \to \mathbb R^4$$
as follows: 
\begin{align*}  & \langle f,g \rangle_{CB^{\delta_{T,r}}_{\gamma_{_{T,r,q}}} (\Omega) } \\ 
= & \left(  \int_{\Omega} (   f_0g_0 +   f_1g_1+   f_2g_2  +  f_3 g_3 ) \gamma_{_{T,r,q}}d\mu , \ 
\int_{\Omega} (   f_0g_1 -  f_1 g_0 +  f_2 g_3-  f_3  g_2 ) \gamma_{_{T,r,q}}d\mu  , \right. \\
  &  \left. \ 
\int_{\Omega} (     f_0g_2 +   f_3 g_1-  f_1 g_3-   f_2g_0   ) \gamma_{_{T,r,q}} d\mu   , \ 
 \int_{\Omega} (    f_0g_3 -  f_1g_2  +   f_2 g_1-   f_3  g_0  ) \gamma_{_{T,r,q}} d\mu \right) ,
\end{align*}
and set $\|(f_0,f_1,f_2,f_3)\|_{CB^{\delta_{T,r}}_{\gamma_{_{T,r,q}}}(\Omega)}^2 = \Pi_0(\langle  f, f \rangle_{CB^{\delta_{T,r}}_{\gamma_{_{T,r,q}}} (\Omega)})$.

At the same time, we will denote by $CB^{q}_{\rho_{_T}} (\Xi) $, the set of four-tuples of elements of  $C^1(\Omega,\mathbb R)$: $(f_0,f_1,f_2,f_3)$ such that satisfy \eqref{equa1} and 
$$\displaystyle \int_{\Xi} |f_k|^2 \rho_{_T} d\mu<+\infty,$$ 
for $k=0,1,2,3$. Recall that $\rho_{_T}$ is given by \eqref{peso1}.   
	
Given $f=(f_0,f_1,f_2,f_3), g=(g_0,g_1,g_2,g_3) \in  CB^{q}_{\rho_{_T}}(\Xi)$ define the real- form
$$\langle \cdot,\cdot \rangle_{CB^{q}_{\rho_{_T}} (\Xi)}: CB^{q}_{\rho_{_T}} (\Xi) \times CB^{q}_{\rho_{_T}} (\Xi) \to \mathbb R^4$$ 
as follows:
\begin{equation*}
\begin{split}
\langle f,g \rangle_{CB^{q}_{\rho_{_T}} (\Xi)}=  \\
\left(\int_{\Xi} (f_0 g_0 + f_1 g_1 + f_2 g_2 + f_3g_3)\rho_{_T}d\mu, \ \int_{\Xi}(f_0g_1-f_1g_0 + f_2g_3 -f_3g_2)\rho_{_T}d\mu, \right. \\
\left. \int_{\Xi} (f_0 g_2 + f_3g_1 - f_1g_3 - f_2 g_0) \rho_{_T} d\mu, \ \int_{\Xi} (f_0 g_3-f_1g_2 + f_2g_1-f_3g_0) \rho_{_T} d\mu \right),
\end{split}
\end{equation*}
and 
$$\|(f_0,f_1,f_2,f_3)\|_{CB^{q}_{\rho_{_T}} (\Xi)}^2 = \Pi_0(\langle f, f \rangle_{CB^{q}_{\rho_{_T}} (\Xi)}).$$  
\end{enumerate}
\end{definition}

\begin{proposition}\label{HilertSpaceCimmino} 
Let $\Omega\subset \mathbb H$ be a domain and set $ (q_0, q_1,q_2,q_3)\in \mathbb R^4$. Then $CB_q(\Omega)$ is a real-Banach space and there exist a four-tuple of real-valued reproducing functions, denoted by $\left(B_0(x,\cdot), B_1(x,\cdot), B_2(x,\cdot), B_3(x,\cdot) \right) \in CB_q(\Omega)$, such that 
\begin{equation}\label{Reprod}
\begin{split}  
\langle \left(B_0 (x,\cdot), -B_1(x,\cdot), -B_2(x,\cdot), -B_3(x,\cdot)\right), \ (f_0, f_1, f_2, f_3)\rangle_{CB_q(\Omega)} = \\
 (f_0(x),f_1(x),f_2(x),f_3 (x))
\end{split}
\end{equation}
or equivalently 
\begin{align*}    \int_{ \Omega }  
    \left( B_0 (x,y )   f_0(y) -  B_1  (x,y )  f_1(y)  - B_2  (x,y )  f_2(y) - 
 B_3  (x,y )    f_3 (y )  \right) d\mu_y =& \ f_0(x) ,\nonumber \\
  \int_{  \Omega }   \left( 
  B_0 (x,y )   f_1(y) +  B_1  (x,y )  f_0(y)  + B_2  (x,y )  f_3(y) - 
 B_3  (x,y )     f_2 (y )  \right) d\mu_y =&  \ f_1(x)  ,\nonumber\\
  \int_{  \Omega }   \left( 
  B_0 (x,y )   f_2(y) -  B_1  (x,y )  f_3(y)  + B_2  (x,y )  f_0(y) + 
 B_3  (x,y )     f_1 (y )  \right) d\mu_y=&  \ f_2(x)   , \nonumber\\
 \int_{  \Omega }   \left( 
  B_0 (x,y )   f_3(y) +  B_1  (x,y )  f_2(y)  - B_2  (x,y )  f_1(y) + 
 B_3  (x,y )     f_0(y )  \right) d\mu_y=& \ f_3(x),
\end{align*} 
for all $f =(f_0,f_1,f_2,f_3) \in CB_q(\Omega)$.

Additionally, if $g_k:\Omega\to \mathbb R$ is a measurable function and 
$$\int_{\Omega} |g_k|^2  d\mu <\infty, \quad  k=0,1,2,3$$   
then the four-tuple of real-valued functions:
\begin{equation} \label{Proyec}
\begin{split}
(h_0(x), h_1(x), h_2(x), h_3(x))= \\ 
\langle \left( B_0 (x,\cdot), -B_1  (x,\cdot), -B_2  (x,\cdot),-B_3(x,\cdot)\right), (g_0,g_1,g_2,g_3)\rangle_{CB_q(\Omega)},
\end{split}
\end{equation}
for all  $x\in \Omega $, belong to $CB_{q}(\Omega)$.
\end{proposition}
\begin{proof} Let $\psi=\{1,e_1,-e_2,e_3\}$, we use the Fact 1. of Proposition \ref{HilertSpace} and denoting  		 
 $${}^{\psi}_{q}\mathcal B_\Omega(x,\cdot) = B_0(x,\cdot) + B_1(x,\cdot) e_1+ B_2(x,\cdot) e_2 + B_3(x,\cdot) e_3,$$
where $ B_n(x,\cdot) $ are the real-valued  functions for $n=0,1,2,3$ we can compute the real components of \eqref{BergKern} and \eqref{BergProj}, which complete the proof.
\end{proof}

\begin{proposition}
Let $\Omega\subset \mathbb H$ be a domain and set $ q, r  \in \mathbb H$. Then $(f_0, f_1,f_2,f_3) \in CB_q(\Omega )$ if and only if  $(g_0, g_1,g_2,g_3) \in CB_r(\Omega )$, where $g_n(x) = e^{\langle q-r, x \rangle } f_n(x)$, for all $x\in \Omega$ for $n=0,1,2,3$.
\end{proposition}
\begin{proof}
The proof is a direct computation and will be omitted.
\end{proof}

\begin{proposition} Given $q,r\in\mathbb R^4$ and let $\Omega$ and $\Xi$ be two conformally equivalent  domains such that $T(\Xi)= \Omega$. Then the real normed function spaces 
$$\left( CB_{\lambda_q}(\Omega), \|\cdot \|_{CB_{\lambda_q}(\Omega)} \right), \left( CB^{\delta_{T,r}}_{\gamma_{_{T,r,q}}}(\Omega), \|\cdot \|_{CB^{\delta_{T,r}}_{\gamma_{_{T,r,q}}} (\Omega)}\right), \left(  CB^{q}_{\rho_{_T}} (\Xi) , \|\cdot \|_{CB^{q}_{\rho_{_T}} (\Xi)}\right)$$ 
are real-Banach function spaces that have four-tuple of reproducing functions, similar property to  \eqref{Reprod}, that satisfies  the projection property \eqref{Proyec}.
\end{proposition}
\begin{proof} The proof follows via Propositions  \ref{HilertSpace} and \ref{ref123456}. 
\end{proof}

\begin{proposition} Given $q,r\in\mathbb R^4$ and let $\Omega$ and $\Xi$ be two conformally equivalent  domains such that $T(\Xi)= \Omega$.  Then $(f_0,f_1,f_2,f_3) \in CB^{\delta_{T,r}}_{\gamma_{_{T,r,q}}} (\Omega)$ if and only if $(g_0, g_1,g_2,g_3) \in CB^{q}_{\rho_{_T}} (\Xi)$, with 
\begin{equation*}
\begin{split}
g_0(x) = e^{\langle r - q, x \rangle} \left[C_T^0 (x)f_0\circ T (x)- C_T^1(x) f_1\circ T(x) -C_T^2 (x)f_2\circ T(x) -C_T^3(x) f_3\circ T(x)\right],\\
g_1(x) = e^{\langle r - q, x \rangle} \left[C_T^1 (x) f_0\circ T(x)+C_T^0(x) f_1\circ T(x)-C_T^3 (x) f_2\circ T(x) + C_T^2(x) f_3\circ T(x) \right],\\
g_2(x) = e^{\langle r - q, x \rangle} \left[C_T^2 (x) f_0\circ T(x) -C_T^1 (x)f_3\circ T(x) +C_T^3 (x) f_1\circ T(x) +C_T^0  (x)f_2\circ T(x)\right],\\
g_3(x) =  e^{\langle r - q  , x \rangle}\left[  C_T^ 3  (x) f_0\circ T(x)-C_T^2 (x) f_1\circ T(x)+C_T^1 (x) f_2\circ T(x) +C_T^0 (x) f_3\circ T(x)\right], 
\end{split}
\end{equation*}
for all $x\in \Omega$, where $C_T$ is expressed by $C_T = C_T^0 + C_T^1 e_1 + C_T^2 e_2 + C_T^3 e_3$ being $C_T^n$ a real-valued function for $n=0,1,2,3$.  What is more, the mapping 
$$(f_0,f_1,f_2,f_3)\mapsto (g_0,g_1,g_2,g_3)$$
is a bijective real-linear operator from $CB^{\delta_{T,r}}_{\gamma_{_{T,r,q}}}(\Omega)$ to $CB^{q}_{\rho_{_T}} (\Xi)$.
\end{proposition}
\begin{proof} Use \eqref{ConfOper} of Proposition \ref{ref123456} and compute its real components, which is our claim.
\end{proof}

\begin{corollary}  Let $\Omega\subset\mathbb H$ be a domain. Then
\begin{enumerate}
\item   
 If $\Omega \subset \{ x \in \mathbb H \  \mid \  \left\langle q  , x \right\rangle <0  \}$
then    $CB_q(\Omega)\subset CB_{\lambda_q}(\Omega)$ as function sets.
  \item  If 
$\Omega \subset \{ x \in \mathbb H \  \mid \  \left\langle q  , x \right\rangle > 0  \}$
then   $CB_{\lambda_q}(\Omega)\subset CB_q(\Omega) $ as function sets. 
 \item If $\Omega$ is a bounded domain then  $ CB_{\lambda_q}(\Omega)= CB_q(\Omega)$ as function sets.
    \end{enumerate}
\end{corollary}
\begin{proof}
It follows from Corollary \ref{relationships} choosing $\psi=\{1,e_1,-e_2,e_3\}$. 
\end{proof}

\small{	
\section*{Declarations}
\subsection*{Funding} Instituto Polit\'ecnico Nacional (grant numbers SIP20220017, SIP20221274) and CONACYT (grant number 1108687).
\subsection*{Conflict of interest} The authors declare that they have no conflict of interest regarding the publication of this paper.
\subsection*{Author contributions} The authors contributed equally to the manuscript and typed, read, and approved the final form of the manuscript, which is the result of an intensive collaboration.
\subsection*{Availability of data and material} Not applicable
\subsection*{Code availability} Not applicable}

\end{document}